\documentclass[a4paper,10pt]{amsart}
\usepackage[colorlinks,linkcolor=blue,citecolor=blue]{hyperref}
\usepackage{latexsym, amssymb, amsmath, amsthm, bbm}
\usepackage[all]{xy}

\usepackage{anysize}\marginsize{25mm}{25mm}{30mm}{30mm}
\allowdisplaybreaks[4]

\def \To{\longrightarrow}

\def \Vec{\operatorname{Vec}}

\def \Z{\mathbb{Z}}

\def \R{\mathcal{R}}

\def \k{\mathbbm{k}}
\def \r{\mathbbm{R}}
\def \1{\mathbf{1}}

\def \Id{\operatorname{Id}}

\def \mod{\operatorname{mod}}
\def \k{\mathbbm{k}}

\numberwithin{equation}{section}

\newtheorem{theorem}{Theorem}[section]
\newtheorem{lemma}[theorem]{Lemma}
\newtheorem{proposition}[theorem]{Proposition}

\newtheorem{remark}[theorem]{Remark}

\begin{document}

\title[The octonions form an Azumaya algebra in Gr-Categories]{The octonions form an Azumaya algebra in certain \\ braided linear Gr-categories$^\dag$} \thanks{\tiny $^\dag$Supported by PCSIRT IRT1264, SRFDP 20130131110001, NSFC 11471186 and SDNSF ZR2013AM022.}
\subjclass[2010]{15A66, 18D10, 16H05}
\keywords{octonions, braided linear Gr-category, Azumaya algebra}

\author[T. Cheng]{Tao Cheng}
\address{School of Mathematics, Shandong University, Jinan 250100, China and School of Mathematics, Shandong Normal University, Jinan 250014, China}  \email{chtao101@163.com}
\author[H.-L. Huang]{Hua-Lin Huang*}\thanks{*Corresponding author.}
\address{School of Mathematics, Shandong University, Jinan 250100, China} \email{hualin@sdu.edu.cn}
\author[Y. Yang]{Yuping Yang}
\address{School of Mathematics, Shandong University, Jinan 250100, China} \email{yupingyang.sdu@gmail.com}
\author[Y. Zhang]{Yinhuo Zhang}
\address{Department WNI, University of Hasselt, Universitaire Campus, 3590 Diepenbeek, Belgium}
\email{yinhuo.zhang@uhasselt.be}

\date{}
\maketitle

\begin{abstract}
 By applying the idea of viewing the octonions as an associative algebra in certain tensor categories, or more precisely as a twisted group algebra by a 2-cochain, we show that the octonions form an Azumaya algebra in some suitable braided linear Gr-categories.
\end{abstract}

\section{Introduction}
As the largest of the four normed division algebras, the octonion algebra is closely related to the quaternion algebra and its natural generalization, Clifford algebras. It is well known that the quaternions form a real Azumaya algebra, while Clifford algebras are  $\Z_2$-graded Azumaya algebras \cite{w}. So a natural question is whether the octonions form an Azumaya algebra in some suitable sense.  The fact that the octonions are \emph{nonassociative} makes this question fairly awkward.

In \cite{am1}, Albuquerque and Majid made a remarkable observation that the octonion algebra is \emph{associative} if it is seen as an algebra in a suitable tensor category. Such viewpoint may even suggest novel methods for the investigation of associative algebras, see Albuquerque and Majid's subsequent study of Clifford algebras \cite{am2}. It turns out their idea helps to provide a solution to the aforementioned question. We observe an explicit relation between the octonions and the real Clifford algebra $\operatorname{Cl}_{0,3}$ via a tensor equivalence between their background tensor categories. This enables us to transfer the structural information, in particular those obtained in \cite{cgo, kn, w}, of the familiar algebra $\operatorname{Cl}_{0,3}$ to the octonions.

To state our main results, first we need to fix some notations. Throughout, let $\r$ denote the reals and $\mathbb{O}$ the octonions. By $\Vec_G^{\Phi}$ we mean the linear graded category (or Gr-category) by a group $G,$ with  associativity constraint given by a 3-cocycle $\Phi$ on $G.$ The linear Gr-category $\Vec_G^{\Phi}$ is said to be braided if there exists a braiding given by a quasi-bicharacter $\R$ of $G$ with respect to $\Phi.$

Let $\Z_2^3$ be the triple direct product of the cyclic group $\Z_2=\{0,1\}$ and $F: \Z_2^3 \times \Z_2^3 \To \r^*$ be the map defined by
\begin{equation}
F(x,y)=(-1)^{x_1x_2y_3+x_1y_2x_3+y_1x_2x_3+\sum_{1\leq i\leq j\leq 3}x_iy_j}, \quad \forall x=(x_1,x_2,x_3), \ y=(y_1,y_2,y_3) \in \Z_2^3.
\end{equation} Then one may define a twisted group algebra $\r_F[\Z_2^3].$ The important observation of Albuquerque and Majid \cite{am1} is that $\mathbb{O} \cong \r_F[\Z_2^3],$ and through the latter the octonions may be naturally viewed as an associative algebra in the $\r$-linear Gr-category $\Vec_{\Z_2^3}^{\partial F},$ where $\partial F$ is the differential of the 2-cochain $F.$ On the other hand, the real Clifford algebra $\operatorname{Cl}_{0,3}$ can also be realized as the twisted group algebra $\r_{F'}[\Z_2^3]$ with
\begin{equation}
F'(x,y)=(-1)^{\sum_{1\leq i\leq j\leq 3}x_iy_j}, \quad \forall x=(x_1,x_2,x_3), \ y=(y_1,y_2,y_3) \in \Z_2^3
\end{equation}
and be seen as an associative algebra in $\Vec_{\Z_2^3}^{\partial F'},$ see \cite{am2}. Note that the Clifford algebra $\operatorname{Cl}_{0,3}$ is associative in the usual sense as $F'$ is a 2-cocycle (hence $\partial F'=0$), while the octonions $\mathbb{O}$ are nonassociative as $F$ is not a 2-cocycle. However, we may identify these two seemingly quite different algebras as one in a suitable sense due to the following immediate fact.

\begin{proposition}
The tensor functor $(\mathcal{F}, \varphi_0, \varphi_2): \Vec_{\Z_2^3}^{\partial F}  \To \Vec_{\Z_2^3}^0$ with \[ \mathcal{F}(U)=U, \quad \varphi_0=\Id_\r, \quad \varphi_2: U \otimes V \To U \otimes V, \ u \otimes v \mapsto (-1)^{x_1x_2y_3+x_1y_2x_3+y_1x_2x_3} u \otimes v  \] for all $U, V$ and $u \in U_x, \ v \in V_y$ is a tensor equivalence and $\mathcal{F}$ maps the algebra $\mathbb{O}$ in $\Vec_{\Z_2^3}^{\partial F}$ to the algebra $\operatorname{Cl}_{0,3}$ in $\Vec_{\Z_2^3}^0.$
\end{proposition}

The reader is referred to \cite{k} for unexplained notations. The proof of the preceding proposition is simple. It is enough to notice that the category $\Vec_{\Z_2^3}^{\partial F}$ is the comodule category of the dual quasi-Hopf algebra $(\r [\Z_2^3], \partial F)$ and $\Vec_{\Z_2^3}^0$ is that of $(\r [Z_2^3], \partial F'),$ and $(\r [\Z_2^3], \partial F)$ is gauge equivalent to $(\r [\Z_2^3], \partial F')$ via the twisting $F^{-1}F'.$ This gauge equivalence induces the prescribed tensor equivalence $\mathcal{F},$ see \cite{am1, k} for more details. Under this tensor equivalence, the algebra $\mathbb{O}$ (with multiplication denoted by $m_\mathbb{O}$) in $\Vec_{\Z_2^3}^{\partial F}$ is mapped to the algebra $\mathcal{F}(\mathbb{O})$ in $\Vec_{\Z_2^3}^0$ whose multiplication is given by the composition \[ \mathcal{F}(\mathbb{O}) \otimes \mathcal{F}(\mathbb{O}) \stackrel{\varphi_2}{\To}  \mathcal{F}(\mathbb{O} \otimes \mathbb{O}) \stackrel{\mathcal{F}(m_\mathbb{O})}{\To} \mathcal{F}(\mathbb{O}). \] More explicitly, for any two elements $u_x \in \mathcal{F}(\mathbb{O})_x, \ u_y \in \mathcal{F}(\mathbb{O})_y, \ x, y \in \Z_2^3,$ the multiplication is \[ u_x u_y = (-1)^{\sum_{1\leq i\leq j\leq 3}x_iy_j}u_{x+y}.\] This is exactly the multiplication of the algebra $\operatorname{Cl}_{0,3}$ in $\Vec_{\Z_2^3}^0.$

\begin{remark}
The tensor equivalence $\mathcal{F}$ also induces a one-to-one correspondence between the set of braidings of $\Vec_{\Z_2^3}^{\partial F}$ and that of $\Vec_{\Z_2^3}^0.$ Given a braiding $\mathcal{R}$ of $\Vec_{\Z_2^3}^{\partial F},$ define $\mathcal{F}(\mathcal{R})$ by \[ \mathcal{F}(\mathcal{R})(x,y)=(-1)^{x_1x_2y_3+x_1y_2x_3+ y_1x_2x_3+y_1y_2x_3+y_1x_2y_3+x_1y_2y_3}\mathcal{R}(x,y), \quad \forall x, y \in \Z_2^3. \] Then $\mathcal{F}(\mathcal{R})$ is a braiding of $\Vec_{\Z_2^3}^0$ and \[ (\mathcal{F}, \varphi_0, \varphi_2): \left(\Vec_{\Z_2^3}^{\partial F}, \mathcal{R}\right)  \To \left(\Vec_{\Z_2^3}^0, \mathcal{F}(\mathcal{R})\right) \] is a braided tensor equivalence, see \emph{ibid}.
\end{remark}

Now $\mathbb{O}$ is an algebra in the braided tensor category $\left(\Vec_{\Z_2^3}^{\partial F}, \mathcal{R}\right).$ Hence, we are in a position to ask: is there a braiding $\mathcal{R}$ such that $\mathbb{O}$ is Azumaya in $\left(\Vec_{\Z_2^3}^{\partial F}, \mathcal{R}\right)$? It is well known that a braided tensor equivalence preserves Azumaya objects, hence by Remark 1.2 our question can be translated to a more familiar situation: find a suitable braiding $\mathcal{R}$ so that the algebra $\operatorname{Cl}_{0,3}$ is Azumaya in $\left(\Vec_{\Z_2^3}^0, \mathcal{F}(\mathcal{R})\right).$  

Clearly, the latter question lies in the well developed theory of group graded Azumaya algebras and Brauer groups, see, e.g., \cite{cgo, kn, w}. It is not surprising that these pioneering works may help to solve our question. In particular, a well known result of Wall \cite{w} indicates that $\operatorname{Cl}_{0,3}$ is an Azumaya algebra in $(\Vec_{\Z_2}^0,\mathcal{P})$ where $\mathcal{P}(x,y)=(-1)^{xy}, \ \forall x,y \in \Z_2.$ Then via a pull-back of the group homomorphism $f: \Z_2^3 \To \Z_2, \quad x \mapsto x_1+x_2+x_3 \ \mod \ 2,$ it seems possible to obtain similar property of $\operatorname{Cl}_{0,3}$ in $\left(\Vec_{\Z_2^3}^0, \mathcal{F}(\mathcal{R})\right),$ and eventually to obtain that of $\mathbb{O}$ in $\left(\Vec_{\Z_2^3}^{\partial F}, \mathcal{R}\right)$ via the braided tensor equivalence $\mathcal{F}.$ This naturally leads to our first result.

\begin{theorem}
$\mathbb{O}$ is an Azumaya algebra in the braided Gr-category $(\Vec_{\Z_2^3}^{\partial F},\R)$ with
\begin{equation}
\R(x,y)=(-1)^{x_1x_2y_3+x_1y_2x_3+ y_1x_2x_3+y_1y_2x_3+y_1x_2y_3+x_1y_2y_3+\sum_{i,j=1}^{3}x_iy_j}, \quad \forall x,y \in \Z_2^3.
\end{equation}
\end{theorem}

The next natural question is that, to what extent the braiding $\R$ is defined by the octonions? Again, this question can be translated to that for $\operatorname{Cl_{0,3}}$ and we may consult \cite{cgo, kn}. It turns out that there are various braidings $\R$ which make $\mathbb{O}$ Azumaya in $(\Vec_{\Z_2^3}^{\partial F},\R)$ and they can be completely classified. This is our second main result.

\begin{theorem}
There are exactly $168$ braidings $\R$ such that $\mathbb{O}$ is an Azumaya algebra in $(Vec_{\Z_2^3}^{\partial F},\R),$ where 
\[ \R(x,y)=(-1)^{x_1x_2y_3+x_1y_2x_3+ y_1x_2x_3+y_1y_2x_3+y_1x_2y_3+x_1y_2y_3+\sum_{i,j=1}^{3}a_{ij}x_iy_j}, \quad \forall x, y\in \Z_2^3 \] with 
$(a_{11},a_{12},a_{13},a_{21},a_{22},a_{23},a_{31},a_{32},a_{33})$ listed in the following table.
\end{theorem}
\tiny
\[\begin{array}{ccccc}
0,0,0,0,0,0,0,0,1&  0,0,0,0,0,0,0,1,0&  0,0,0,0,0,0,1,0,0&  0,0,0,0,0,0,1,1,1&  0,0,0,0,0,1,0,0,0\\
0,0,0,0,0,1,0,1,1&  0,0,0,0,0,1,1,0,0&  0,0,0,0,0,1,1,1,1&  0,0,0,0,1,0,0,0,0&  0,0,0,0,1,0,0,1,1\\
0,0,0,0,1,0,1,0,0&  0,0,0,0,1,0,1,1,1&  0,0,0,0,1,1,0,0,1&  0,0,0,0,1,1,0,1,0&  0,0,0,0,1,1,1,0,0\\
0,0,0,0,1,1,1,1,1&  0,0,0,1,0,0,0,0,0&  0,0,0,1,0,0,0,0,1&  0,0,0,1,0,0,0,1,0&  0,0,0,1,0,0,0,1,1\\
0,0,0,1,1,1,0,0,0&  0,0,0,1,1,1,0,0,1&  0,0,0,1,1,1,0,1,0&  0,0,0,1,1,1,0,1,1&  0,0,1,0,0,0,0,0,0\\
0,0,1,0,0,0,0,1,0&  0,0,1,0,0,0,1,0,1&  0,0,1,0,0,0,1,1,1&  0,0,1,0,0,1,0,0,1&  0,0,1,0,0,1,0,1,1\\
0,0,1,0,0,1,1,0,1&  0,0,1,0,0,1,1,1,1&  0,0,1,0,1,0,0,0,0&  0,0,1,0,1,0,0,1,0&  0,0,1,0,1,0,1,0,1\\
0,0,1,0,1,0,1,1,1&  0,0,1,0,1,1,0,0,1&  0,0,1,0,1,1,0,1,1&  0,0,1,0,1,1,1,0,1&  0,0,1,0,1,1,1,1,1\\
0,0,1,1,0,0,0,0,0&  0,0,1,1,0,0,0,0,1&  0,0,1,1,0,0,0,1,0&  0,0,1,1,0,0,0,1,1&  0,0,1,1,1,0,0,0,0\\
0,0,1,1,1,0,0,0,1&  0,0,1,1,1,0,0,1,0&  0,0,1,1,1,0,0,1,1&  0,1,0,0,0,0,0,0,0&  0,1,0,0,0,0,0,0,1\\
0,1,0,0,0,0,1,0,0&  0,1,0,0,0,0,1,0,1&  0,1,0,0,0,1,0,0,0&  0,1,0,0,0,1,0,0,1&  0,1,0,0,0,1,1,0,0\\
0,1,0,0,0,1,1,0,1&  0,1,0,0,1,0,0,1,0&  0,1,0,0,1,0,0,1,1&  0,1,0,0,1,0,1,0,0&  0,1,0,0,1,0,1,0,1\\
0,1,0,0,1,1,0,1,0&  0,1,0,0,1,1,0,1,1&  0,1,0,0,1,1,1,0,0&  0,1,0,0,1,1,1,0,1&  0,1,0,1,1,0,0,0,0\\
0,1,0,1,1,0,0,0,1&  0,1,0,1,1,0,0,1,0&  0,1,0,1,1,0,0,1,1&  0,1,0,1,1,1,0,0,0&  0,1,0,1,1,1,0,0,1\\
0,1,0,1,1,1,0,1,0&  0,1,0,1,1,1,0,1,1&  1,0,0,0,0,0,0,0,0&  1,0,0,0,0,0,0,1,0&  1,0,0,0,0,0,1,0,1\\
1,0,0,0,0,0,1,1,1&  1,0,0,0,0,1,0,0,0&  1,0,0,0,0,1,0,1,1&  1,0,0,0,0,1,1,0,0&  1,0,0,0,0,1,1,1,1\\
1,0,0,0,1,1,0,1,0&  1,0,0,0,1,1,0,1,1&  1,0,0,0,1,1,1,0,0&  1,0,0,0,1,1,1,0,1&  1,0,0,1,0,0,0,1,0\\
1,0,0,1,0,0,0,1,1&  1,0,0,1,0,0,1,0,0&  1,0,0,1,0,0,1,0,1&  1,0,0,1,1,0,0,0,0&  1,0,0,1,1,0,0,1,1\\
1,0,0,1,1,0,1,0,0&  1,0,0,1,1,0,1,1,1&  1,0,0,1,1,1,0,0,0&  1,0,0,1,1,1,0,1,0&  1,0,0,1,1,1,1,0,1\\
1,0,0,1,1,1,1,1,1&  1,0,1,0,0,0,0,0,1&  1,0,1,0,0,0,0,1,0&  1,0,1,0,0,0,1,0,0&  1,0,1,0,0,0,1,1,1\\
1,0,1,0,0,1,0,0,1&  1,0,1,0,0,1,0,1,1&  1,0,1,0,0,1,1,0,1&  1,0,1,0,0,1,1,1,1&  1,0,1,0,1,0,0,1,0\\
1,0,1,0,1,0,0,1,1&  1,0,1,0,1,0,1,0,0&  1,0,1,0,1,0,1,0,1&  1,0,1,1,0,0,0,1,0&  1,0,1,1,0,0,0,1,1\\
1,0,1,1,0,0,1,0,0&  1,0,1,1,0,0,1,0,1&  1,0,1,1,1,0,0,0,1&  1,0,1,1,1,0,0,1,0&  1,0,1,1,1,0,1,0,0\\
1,0,1,1,1,0,1,1,1&  1,0,1,1,1,1,0,0,1&  1,0,1,1,1,1,0,1,1&  1,0,1,1,1,1,1,0,1&  1,0,1,1,1,1,1,1,1\\
1,1,0,0,0,1,0,0,0&  1,1,0,0,0,1,0,0,1&  1,1,0,0,0,1,1,0,0&  1,1,0,0,0,1,1,0,1&  1,1,0,0,1,0,0,0,0\\
1,1,0,0,1,0,0,1,0&  1,1,0,0,1,0,1,0,1&  1,1,0,0,1,0,1,1,1&  1,1,0,0,1,1,0,0,1&  1,1,0,0,1,1,0,1,0\\
1,1,0,0,1,1,1,0,0&  1,1,0,0,1,1,1,1,1&  1,1,0,1,0,0,0,0,0&  1,1,0,1,0,0,0,0,1&  1,1,0,1,0,0,1,0,0\\
1,1,0,1,0,0,1,0,1&  1,1,0,1,1,0,0,0,1&  1,1,0,1,1,0,0,1,0&  1,1,0,1,1,0,1,0,0&  1,1,0,1,1,0,1,1,1\\
1,1,0,1,1,1,0,0,0&  1,1,0,1,1,1,0,1,0&  1,1,0,1,1,1,1,0,1&  1,1,0,1,1,1,1,1,1&  1,1,1,0,0,0,0,0,0\\
1,1,1,0,0,0,0,0,1&  1,1,1,0,0,0,1,0,0&  1,1,1,0,0,0,1,0,1&  1,1,1,0,1,0,0,0,0&  1,1,1,0,1,0,0,1,1\\
1,1,1,0,1,0,1,0,0&  1,1,1,0,1,0,1,1,1&  1,1,1,0,1,1,0,0,1&  1,1,1,0,1,1,0,1,1&  1,1,1,0,1,1,1,0,1\\
1,1,1,0,1,1,1,1,1&  1,1,1,1,0,0,0,0,0&  1,1,1,1,0,0,0,0,1&  1,1,1,1,0,0,1,0,0&  1,1,1,1,0,0,1,0,1\\
1,1,1,1,1,0,0,0,0&  1,1,1,1,1,0,0,1,1&  1,1,1,1,1,0,1,0,0&  1,1,1,1,1,0,1,1,1&  1,1,1,1,1,1,0,0,1\\
1,1,1,1,1,1,0,1,1&  1,1,1,1,1,1,1,0,1&  1,1,1,1,1,1,1,1,1&  &
\end{array}\]
\normalsize
\section{Preliminaries}

\subsection{Braided linear Gr-categories}
Let $G$ be a group and $\k$ a field. A $\k$-linear Gr-category over $G$ is a tensor category $\Vec_G^\Phi$ which consists of finite-dimensional $\k$-spaces graded by $G$ with the usual tensor product and with associativity constraint given by a normalized 3-cocycle $\Phi$ on $G,$ i.e. a function $\Phi: G \times G \times G \rightarrow \k^*$ such that for all $x,y,z,t \in G$
\[ \Phi(y,z,t)\Phi(x,yz,t)\Phi(x,y,z)=\Phi(x,y,zt)\Phi(xy,z,t), \quad \mathrm{and} \quad \Phi(x,e,y)=1, \]
where $e$ is the unit of $G.$ A Gr-category $\Vec_G^\Phi$ is said to be braided, if there exists a braiding given by a quasi-bicharacter $\R$ with respect to $\Phi,$ that is a function $\R: G \times G \rightarrow \k^*$ satisfying
\begin{gather}
\R(xy, z)=\R(x , z)\R(y ,z)\frac{\Phi(z,x,y)\Phi(x,y,z)}{\Phi(x,z,y)},\\
\R(x , yz)=\R(x , y)\R(x , z)\frac{\Phi(y,x,z)}{\Phi(y,z,x)\Phi(x,y,z)}
\end{gather} for all $x,y,z \in G.$ Note that $\Vec_G^\Phi$ is braided only if $G$ is abelian.

Given a function $F:G\times G \rightarrow \k^*$ with $F(e,x)=F(x,e)=1,$ i.e. a 2-cochain, define
$$\partial F=\frac{F(x,y)F(xy,z)}{F(x,yz)F(y,z)},\ \ \R(x,y)=\frac{F(x,y)}{F(y,x)}$$ for all $x,y,z\in G.$ Then $\partial F$ is a 3-coboundary, i.e. a 3-cocycle cohomologous to 0, on $G$ and $\R$ is a quasi-bicharacter with respect to $\partial F.$

For more details on braided linear Gr-categories, the reader is referred to \cite{hly1,hly2}.

\subsection{Algebras in Gr-categories}
An algebra in the Gr-category $\Vec_G^\Phi$ is an object $A$ with a multiplication morphism $m: A \otimes A \to A$ and a unit morphism $u: \k \to A$ satisfying the associativity and the unitary conditions of usual associative algebras but expressed in diagrams of the category $\Vec_G^\Phi.$ More precisely, an algebra $A$ in $\Vec_G^\Phi$ is a finite-dimensional $G$-graded space $\oplus_{g \in G} A_g$ with a multiplication $\cdot$ such that \begin{gather} A_g \cdot A_h
\subseteq A_{gh}, \\ (a \cdot b) \cdot c = \Phi(|a|,|b|,|c|) a
\cdot (b \cdot c) \end{gather} for all homogeneous elements $a,b,c
\in A,$ where $|a|$ denotes the degree of $a.$ There is also a unit
element $\1$ in $A$ such that \begin{equation} \1 \cdot a=a=a \cdot
\1
\end{equation} for all $a \in A.$ Further, if the Gr-category $\Vec_G^\Phi$ is braided with braiding $\R,$ then we say that two homogeneous elements $a,b \in A$ are commutative if $a\cdot b=\R(|b|,|a|)b\cdot a.$ If any two homogeneous elements commute, then we say $A$ is a commutative algebra in $(\Vec_G^\Phi,\R).$

\subsection{Twisted group algebras}Let $G$ be a group and $F:G\times G \rightarrow \k^*$ a 2-cochain, then we can define a new algebra $\k_F[G]$ which has the same vector space as the group algebra $\k [G]$ but a different product, namely
\begin{equation}
 x \cdot y=F(x,y)xy, \ \forall x,y\in G.
 \end{equation}
Clearly $\k_F[G]$ is an associative commutative algebra in the braided linear Gr-category $(\Vec_G^{\partial F},\R)$ with $\R(x,y)=\frac{F(x,y)}{F(y,x)}$ for all $x,y \in G.$

Thanks to the observation of Albuquerque and Majid \cite{am1}, the octonion algebra $\mathbb{O}$ can be realized as a twisted group algebra $\r_F[\Z_2^3]$ with
\begin{equation*}
F(x,y)=(-1)^{x_1x_2y_3+x_1y_2x_3+y_1x_2x_3+\sum_{1\leq i\leq j\leq 3}x_iy_j}
\end{equation*}
where $x=(x_1,x_2,x_3),y=(y_1,y_2,y_3)\in \Z_2^3$. As a consequence, $\mathbb{O}$ is \emph{associative} in $\Vec_{\Z_2^3}^{\partial F}$ and \emph{commutative} in $(\Vec_{\Z_2^3}^{\partial F},\R)$ with $\R(x,y)=\frac{F(x,y)}{F(y,x)}.$

\subsection{Azumaya algebras in $\left(\Vec_G^0, \mathcal{R}\right)$, or graded Azumaya algebras}
The definition of Azumaya algebras in an arbitrary braided monoidal category can be found in \cite{oz}. For our purpose, we only need to recall the definition of Azumaya algebras in braided linear Gr-categories of form $\left(\Vec_G^0, \mathcal{R}\right)$, or equivalently, group graded Azumaya algebras defined and studied in \cite{cgo, kn, w}. 

Let $A$ be an algebra in $\left(\Vec_G^0, \mathcal{R}\right).$ Call $A$ \emph{simple in $\left(\Vec_G^0, \mathcal{R}\right)$} if it has no nontrivial ideals in $\left(\Vec_G^0, \mathcal{R}\right).$ In other words, $A$ is a $G$-graded algebra and it has no proper graded two-sided ideals. The left center $Z^l(A)$ and the right center $Z^r(A)$ of $A$ in $\left(\Vec_G^0, \mathcal{R}\right)$ are defined by
\begin{equation}
Z^l(A)=\operatorname{span}\{\mathrm{homogeneous} \ a\in A|\forall \ \mathrm{homogeneous} \ x\in A,x\cdot a=\R(|x|,|a|)a\cdot x\},
\end{equation}
\begin{equation}
Z^r(A)=\operatorname{span}\{\mathrm{homogeneous} \ a\in A|\forall \ \mathrm{homogeneous} \ x\in A,a\cdot x=\R(|a|,|x|)x\cdot a\}.
\end{equation}
Call $A$ left central (resp. right central) in $\left(\Vec_G^0, \mathcal{R}\right)$ if $Z^l(A)$ (resp. $Z^r(A)$) is equal to $\k,$ and \emph{central in $\left(\Vec_G^0, \mathcal{R}\right)$} if both $Z^l(A)$ and $Z^r(A)$ are equal to $\k$. Finally, an algebra $A$ in $\left(\Vec_G^0, \mathcal{R}\right)$ is called an \emph{Azumzya algebra in $\left(\Vec_G^0, \mathcal{R}\right)$} if it is both central and simple in $\left(\Vec_G^0, \mathcal{R}\right).$

\section{The proofs of Theorems 1.3 and 1.4}
From now on, Gr-categories are $\r$-linear, $G=\Z_2^3$ and $F, F'$ are the functions defined respectively by (1.1) and (1.2).

\begin{lemma} The following
\begin{equation}\small \{\R(x,y)=(-1)^{x_1x_2y_3+x_1y_2x_3+ y_1x_2x_3+y_1y_2x_3+y_1x_2y_3+x_1y_2y_3+\sum_{i,j=1}^{3}a_{ij}x_iy_j}|x,y\in G, a_{ij}=0,1\}\end{equation} is a complete set of braidings of $\Vec_G^{\partial F}.$
\end{lemma}

\begin{proof}
By Remark 1.2, there is a one-to-one correspondence between the set of braidings in $\Vec_G^{\partial F}$ and $\Vec_G^0.$ As for the latter, the braidings $\R$ are precisely the usual bicharacter on $G,$ that is,  
\begin{equation}
\R(x,y)=(-1)^{\sum_{i,j=1}^{3}a_{ij}x_iy_j}, \quad \forall x,y\in G, \ \ a_{ij} \in \{0,1\}.
\end{equation}
Then by the pull-back along the functor $\mathcal{F},$ we get the corresponding set (3.1) of braidings of $\Vec_G^{\partial F}.$
\end{proof}

\begin{proof}[\bf{The proof of Theorem 1.3}]
Let $\R$ be the braiding as given by (1.3). Then the corresponding braiding $\mathcal{F}(\mathcal{R})$ of $\Vec_G^0$ is given by \[ \mathcal{F}(\mathcal{R})(x,y)=(-1)^{\sum_{i,j=1}^{3}x_iy_j}, \quad \forall x,y \in G. \] By the comment after Remark 1.2, it suffices to prove that $\operatorname{Cl_{0,3}}$ is Azumaya in $\left(\Vec_G^0, \mathcal{F}(\mathcal{R})\right).$ We verify that $\operatorname{Cl_{0,3}}$ is both central and simple in $\left(\Vec_G^0, \mathcal{F}(\mathcal{R})\right).$

Note that each non-zero homogeneous element of $\operatorname{Cl_{0,3}}$ is invertible, so clearly $\operatorname{Cl_{0,3}}$ has no proper graded ideals and hence it is simple in $\left(\Vec_G^0,\mathcal{F}(\R)\right).$ To prove the theorem, it remains to prove that $\operatorname{Cl_{0,3}}$ is central in $\left(\Vec_G^0, \mathcal{F}(\mathcal{R})\right).$ Note that the braiding $\mathcal{F}(\mathcal{R})$ is symmetric, i.e., $\mathcal{F}(\mathcal{R})(x,y)\mathcal{F}(\mathcal{R})(y,x)=1$ for all $x,y\in G.$ So we have $Z^l(\operatorname{Cl_{0,3}})=Z^r(\operatorname{Cl_{0,3}})$. In the following, let $Z$ denote $Z^l(\operatorname{Cl_{0,3}})=Z^r(\operatorname{Cl_{0,3}})$ and we prove $Z=\r$ by direct verification.

For $x=(1,0,0),(0,1,0),(0,0,1),(1,1,1)$, we have $\mathcal{F}(\R)(x,x)=-1,$ so $x\cdot x\neq \mathcal{F}(\R)(x,x)x\cdot x,$ hence $(1,0,0),(0,1,0),(0,0,1),(1,1,1)$ do not belong to $Z.$ For $x=(0,1,1),$ let $y=(0,1,0)$ and we have $\mathcal{F}(\R)(x,y)=1.$ But $x\cdot y=F'(x,y)xy=-y\cdot x,$ hence $(0,1,1)$ is not in $Z.$ Similarly, for $x=(1,0,1),(1,1,0)$, we find $y=(1,0,0)$ such that $x\cdot y\neq \mathcal{F}(\R)(x,y)y\cdot x$, hence $(1,0,1),(1,1,0)$ are not in $Z.$ Consequently, $Z=\r.$

The theorem is proved.
\end{proof}

\begin{proof}[\bf{The proof of Theorem 1.4}]
Again by Remark 1.2 and the comment thereafter, it is enough to determine all the braidings $\R$ of $\Vec_G^0$ such that $\operatorname{Cl_{0,3}}$ is Azumaya in $\left(\Vec_G^0, \R \right).$ As $\operatorname{Cl_{0,3}}$ is always simple in $\left(\Vec_G^0, \R \right),$ it suffices to find $\R$ of form (3.2) such that for any homogeneous element $x \in \operatorname{Cl_{0,3}}$ which is not a scalar multiple of the identity, there exists at least one homogeneous element $y \in \operatorname{Cl_{0,3}}$ such that $x \cdot y \neq \R(x,y) y\cdot x$ or $y \cdot x \neq \R(y,x) x \cdot y.$ The set of such $\R$ can be completely determined with a help of a simple program of Matlab, see the appendix. Then by pulling back along $\mathcal{F},$ we get the desired set of braidings of $\Vec_G^{\partial F}$ as presented in Theorem 1.4. 
\end{proof}

Finally, we remark that it is possible to classify all the braided linear Gr-categories in which the octonion algebra is Azumaya by our method with a help of the result of Elduque \cite{e} which provides all the gradings on the octonions by finite groups and thus all the possibilities of $\mathbb{O}$ as an algebra in braided linear Gr-categories. As this is only a matter of computations, we do not include further details.

\section{Appendix}
\subsection{The method of the computation of bradings.}

 By (3.2), there are 512 possibilities of braidings $\R$ for $\Vec_G^0.$ Firstly we can determine the braiding $\R$ such that $\operatorname{Cl_{0,3}}$ is not central in $\left(\Vec_G^0,\R\right).$ Take the elements of $G$ as a basis of $\operatorname{Cl_{0,3}}.$ Then $\operatorname{Cl_{0,3}}$ is not central in $\left(\Vec_G^0,\R\right)$ if and only if there is an $x \in G-(0,0,0)$ which is a left or right central element. Clearly, this condition can be explicitly expressed by systems of linear equations. Then by subtracting the braidings obtained in the previous step, we get all of those braidings $\R$ such that  
$\operatorname{Cl_{0,3}}$ is central in $\left(\Vec_G^0,\R\right).$

\subsection{The program of the computation} The following is a program of Matlab.

\noindent $AA=[1 0 0;0 1 0; 0 0 1; 0 1 1;1 0 1; 1 1 0;1 1 1];$\\
$m=0;$  
$A=\mathrm{zeros}(9,1);$\\
$\mathrm{Final}\_M=\mathrm{zeros}(9,1);$\\
$M=1:512;$\\
$N=[];$\\
while size$(A,2)<512$\\
       $B=\mathrm{randint}(3,3);$ 
       $A(:,\mathrm{size}(A,2)+1)=B(:);$\\ 
       $C=A';$\\
       $ D=\mathrm{unique}(C,'\mathrm{rows}');$\\ 
      $ A=D';$\\
end \\
 for $i=1:7$\\
     $x=AA(i,:);$  
 for $k=1:512 $ 
     $A\_k=A(:,k);$\\
     $A\_3=\mathrm{reshape}(A_k,3,3);$\\
     $R=r\_\mathrm{fun}(x,x,A_3);$\\
 $if R==1 $  
     $R=0;$\\
     $for j=1:7 $ 
     $y=AA(j,:);$\\
     $\mathrm{index}=y~=x;$\\
     $\mathrm{ddd}=\mathrm{index}(1)+\mathrm{index}(2)+\mathrm{index}(3);$ 
     if $\mathrm{ddd}>0  $\\
        $R=R+r\_\mathrm{fun}(x,y,A_3);$\\
        end\\
     end\\
     if $R==-6 $ 
         $m=m+1$ 
     $\mathrm{Final}\_M(:,\mathrm{size}(\mathrm{Final}\_M,2)+1)=A(:,k);$\\
     $N=[N k];$\\
     $N=\mathrm{unique}(N);$\\
     end\\
   end\\
  end\\
 end

\end{document}